\documentclass[11pt,a4paper]{amsart}
\usepackage[all]{xy}
\usepackage{graphicx}
\usepackage{amssymb}
\usepackage{theoremref}
\usepackage{comment}
\usepackage{tabularx}
\usepackage{fancyhdr}
\usepackage{calc} 
\usepackage{graphicx}
\pdfoutput=1

\usepackage[applemac]{inputenc}
\usepackage[T1]{fontenc}
\usepackage{amsmath}
\usepackage{amsfonts}

\usepackage{tikz}

\usepackage{lipsum}
\numberwithin{equation}{section}

\title{On Frobenius (completed) orbit categories}
\author{ALFREDO N\'AJERA CH\'AVEZ}
\address {Max-Planck-Institut Für Mathematik\\
          Vivatsgasse 7,
                   53111 Bonn\\
                   Germany.}
\email {alfredo.najera-chavez@imj-prg.fr}

\newcommand{\ie}{{\em i.e. }}
\newcommand{\cf}{{\em cf.}\ }

\newcommand{\ul}[1]{\underline{#1}}
\newcommand{\ol}[1]{\overline{#1}}

\newtheorem{theorem}{Theorem}
\newtheorem{lemma}[theorem]{Lemma}
\newtheorem{proposition}[theorem]{Proposition}
\newtheorem{corollary}[theorem]{Corollary}

\newtheorem{lemma and definition}[theorem]{Lemma and Definition}

\theoremstyle{definition}

\newtheorem{notation}[theorem]{Notation}
\newtheorem*{theorem*}{Theorem}

\newtheorem{remark}[theorem]{Remark}
\newtheorem{example}[theorem]{Example}
\newtheorem{definition}[theorem]{Definition}

\newtheorem{assumption}[theorem]{Assumption}
\newtheorem{warning}[theorem]{Warning}

\newcommand{\opname}[1]{\operatorname{\mathsf{#1}}}

\renewcommand{\mod}{\opname{mod}\nolimits}

\newcommand{\proj}{\opname{proj}\nolimits}

\newcommand{\Mod}{\opname{Mod}\nolimits}

\newcommand{\op}{^{op}}

\newcommand{\corb}{\ \widehat{\!\! /} }

\newcommand{\rep}{\opname{rep}\nolimits}

\newcommand{\eff}{\opname{eff}\nolimits}

\newcommand{\pretr}{\opname{pretr}\nolimits}

\newcommand{\ac}{\mathcal{A}c}

\newcommand{\caf}{\ca \corb F}
\newcommand{\paf}{\cp \corb  F}
\newcommand{\eaf}{\ce \corb F}

\newcommand{\colim}{\opname{colim}}
\newcommand{\cok}{\opname{cok}\nolimits}

\newcommand{\Z}{\mathbb{Z}}

\newcommand{\px}{\phantom{x}}

\newcommand{\Hom}{\opname{Hom}}

\newcommand{\Ext}{\opname{Ext}}

\newcommand{\ten}{\otimes}
\newcommand{\lten}{\overset{\boldmath{L}}{\ten}}

\newcommand{\gpr}{\operatorname{gpr}\nolimits}

\newcommand{\fun}{\opname{fun}}
\newcommand{\inv}{\opname{inv}}

\newcommand{\ca}{{\mathcal A}}
\newcommand{\cb}{{\mathcal B}}
\newcommand{\cc}{{\mathcal C}}
\newcommand{\cd}{{\mathcal D}}
\newcommand{\ce}{{\mathcal E}}

\newcommand{\ch}{{\mathcal H}}

\newcommand{\cp}{{\mathcal P}}
\newcommand{\cR}{{\mathcal R}}

\newcommand{\cs}{{\mathcal S}}
\newcommand{\ct}{{\mathcal T}}

\newcommand{\cv}{{\mathcal V}}

\begin{document}

\maketitle

\begin{abstract}
Let $\ce$ be a Frobenius category, ${\mathcal P}$ its subcategory of projective objects and $F:{\mathcal E} \to {\mathcal E}$ an exact automorphism. We prove that there is a fully faithful functor from the orbit category ${\mathcal E}/F$ into $\operatorname{gpr}({\mathcal P}/F)$, the category of finitely-generated Gorenstein-projective modules over ${\mathcal P}/F$. We give sufficient conditions to ensure that the essential image of ${\mathcal E}/F$ is an extension-closed subcategory of $\operatorname{gpr}({\mathcal P}/F)$. If ${\mathcal E}$ is in addition Krull-Schmidt, we give sufficient conditions to ensure that the completed orbit category ${\mathcal E} \ \widehat{\!\! /}  F$ is a Krull-Schmidt Frobenius category. Finally, we apply our results on completed orbit categories to the context of Nakajima categories associated to Dynkin quivers and sketch applications to cluster algebras.
\end{abstract}

\tableofcontents

\section{Introduction}
Let $\ce$ be an additive category and $F: \ce \to \ce$ an automorphism. The category of orbits associated to this data was first introduced by Cibils and Marcos in \cite{Cibils Marcos}. It was further studied by Asashiba in \cite{Asashiba, Asashiba 2} and by Keller in \cite{Keller triang, Keller triang corrections}. By definition, the orbit category $\ce/F$ has the same objects as $\ce$, the set of morphisms from an object $X$ to an object $Y$ is given by 
\begin{equation}
\label{usual orbit category}
\ce / F(X,Y)= \bigoplus_{l\in \Z} \ce(X,F^l(Y))
\end{equation}	
and the composition of morphisms is defined in a natural way (see (\ref{composition})). Clearly $\ce/F$ is still an additive category and the canonical projection $p:\ce \to \ce/F$ is an additive functor. Now suppose that $\ce$ is a Frobenius category and that $F$ is an exact functor. We are interested in the following question:
\\

\emph{Is there an exact structure on the orbit category such that $\ce/F$ becomes a Frobenius category and the canonical projection an exact functor?}
\\

In this article we give sufficient conditions for the answer to be positive. Recall that Frobenius categories and triangulated categories are closely related: the stable category $\ul{\ce}$ is canonically triangulated and $F$ induces a triangulated automorphism $\ul{F}$ on $\ul{\ce}$. Moreover, the analogous question for triangulated categories was already studied in \cite{Keller triang} where the author defined the triangulated hull of the orbit category. 
In a similar way we show that $\ce/F$ embeds into a certain ambient exact category and give sufficient conditions to ensure that $\ce/F$ is an extension-closed subcategory of the ambient category. To be more precise, let $\cp$ be the full subcategory of $\ce$ determined by its projective objects and $\gpr(\cp/F)$ be the category of finitely-generated Gorenstein-projective modules over $\cp /F$ (see \thref{gpr}). The category $\gpr(\cp/F)$ is an exact category, it is even a Frobenius category. Inspired by a result of Chen \cite{Chen}, we prove that there is a full and faithful functor $\ce/F \hookrightarrow \gpr(\cp/F)$. We prove the following theorem which is \thref{main theorem usual} of this note.

\begin{theorem*}
Suppose that $\ul{\ce}/\ul{F}$ is equivalent to its triangulated hull. If $\ce/F$ has split idempotents, then $\ce/F$ is closed under extensions in $\gpr(\cp/F)$. Moreover, the induced exact structure on $\ce /F$ makes the canonical projection $\ce \to \ce/F$ exact and makes $\ce/F$ a Frobenius category whose stable category is triangle equivalent to $\ul{\ce}/\ul{F}$.
\end{theorem*}

The ambient Frobenius category $\gpr(\cp/F)$ \emph{should not} be considered as the exact (or Frobenius) hull of $\ce /F$ since it  may be too large. Still, $\gpr(\cp/F)$ can be considered as a canonical ambient category for $\ce/F$ since  the functor $\ce/F \hookrightarrow \gpr(\cp/F)$ is induced (in a sense made precise in \thref{restricted Yoneda}) by the Yoneda functor $\ce/F \hookrightarrow \Mod(\ce/F)$. Even more, the triangulated structure on the stable category $\ul{\gpr}(\cp/F)$ and on the triangulated hull of $\ul{\ce}/\ul{F}$ are compatible in the following sense: the inclusion $\ce/ F \hookrightarrow \gpr(\cp/F)$ induces a fully faithful triangulated functor from the triangulated hull of $\ul{\ce}/\ul{F}$ into $\ul{\gpr}(\cp/F)$ (see \thref{triangulated functor}).

Orbit categories have appeared (perhaps sometimes in an implicit way) in the work of many mathematicians (see for instances \cite{Happel derived} \cite{BMRRT}, \cite{Keller triang}, 
\cite{Fu root}, \cite{Qin} and \cite{Scherotzke}). One of our main motivations for studying orbit categories of Frobenius categories comes from representation theory, and more concretely, from the additive categorification of acyclic cluster algebras introduced in \cite{BMRRT} (see also \cite{clusters 1} and \cite{Keller survey}). In particular, we are interested in the case where $\ce$ is in addition a Krull-Schmidt category. In general $\ce/F$ fails to be Krull-Schmidt. Yet, \emph{under certain finiteness conditions} on $F$ (stated explicitly in Section 6) we are able to give sufficient conditions to prove the following theorem which is \thref{main theorem} of this note.
\begin{theorem*}
If $\ce$ is Krull-Schmidt (and $F$ is as discribed above) then the completed orbit category $\eaf$ admits the structure of a Frobenius category which makes the canonical functor $\ce \to \eaf$ exact and whose stable category is triangle equivalent to $\ul{\ce}/\ul{F}$.
\end{theorem*}
The \emph{completed orbit category} is defined just as the usual orbit category by replacing the direct sum in (\ref{usual orbit category}) by the direct product. The composition formula (\ref{composition}) of usual orbit categories defines a composition in the completed orbit category provided that, for every pair of objects $X$ and $Y$, the group $\ce(X,F^l (Y))$ vanishes for $l \ll 0$. This last theorem will allow us to give an explicit categorification of families of finite-type cluster algebras with coefficients. In particular, we obtain a categorification of all skew-symmetric finite-type cluster algebras with universal coefficients. We would like to stress that completed orbit categories already appeared in \cite{Igusa Todorov} where they are used to define continuous cluster categories.



This article is organized as follows. In the Section 2 we survey Keller's construction of the triangulated hull associated to the orbit category of a triangulated category. In section~3 we recall a Theorem of Chen which shows that any Frobenius category is equivalent to an extension-closed exact subcategory of the Frobenius category formed by Gorenstein-projective modules over some additive category. This theorem will allow us to define the embedding of $\ce /F$ into $\gpr(\cp/F)$. In section 4 we prove some general results on usual and completed orbit categories which will be used intensively. In section 5 we prove the compatibility of the triangulated structure of  $\ul{\gpr}(\cp/F)$ and of the triangulated hull of $\ul{\ce}/\ul{F}$. In section 6 we give a proof of the theorems stated above. In section 7 we apply our results  on completed orbit categories to the context of Nakajima categories associated to Dynkin quivers to introduce explicit categorifications of families of finite-type cluster algebras with coefficients. In particular, we obtain a categorification of all skew-symmetric finite-type cluster algebras with universal coefficients.

\subsection*{Acknowledgments}
The work in this article is exposed in my Ph.D. thesis, supervised by Professors Bernhard Keller and Lauren Williams. I would like to thank them for their guidance and patience. I am grateful to Raika Dehy, Patrick Le Meur, Yann Palu and Pierre-Guy Plamondon for the useful discussions. The final version of this note was written at the Max Planck Institut for Mathematics, Bonn, during the 2015 fall semester. The author is deeply grateful to this institution for the financial support and for providing ideal working conditions.

\section{Reminders on dg categories and their orbit categories}

Throughout this chapter, we will freely use the theory of exact categories first introduced by Quillen in \cite{Quillen}. Our main reference for this theory is the refined treatment presented in \cite[Appendix A]{Keller chain} and the systematic study of \cite{Buhler}. We will also use the basic facts about Frobenius categories. The reader is referred to \cite{Happel book} for a treatment of this topic. The set of morphisms between two objects $X$ and $Y$ of a category $\ca$ is denoted by $\ca(X,Y)$. If $k$ is a ring and $\ca$ an additive $k$-category, a right $\ca$-module is by definition a $k$-linear functor $M:\ca^{\text{op}} \rightarrow \Mod k$, where $\Mod k$ is the category of all right $k$-modules. The morphism space between two $\ca$-modules $L$ and $M$ is denoted by $\Hom_{\ca}(L,M)$ or simply $\Hom (L,M)$ when there is no risk of confusion.

\subsection{Pretriangulated dg categories.}
In this section we recall some facts about dg categories and introduce notation. Our main reference for these results are \cite{Keller on dg} and \cite{Drinfeld}. We work over an arbitrary field $k$. In this section, all categories, dg categories, functors, dg functors, etc. are assumed to be $k$-linear.

A dg category $\cb$ is a category whose morphism spaces have the structure of a \emph{differential graded $k$-module}, or equivalently, a \emph{complex of $k$-modules}. For a dg category $\cb$ we denote by $Z^0(\cb)$ the category with the same objects as $\cb$, and with morphisms $Z^0(\cb)(X,Y)= Z^0(\cb(X,Y))$. The category $H^0(\cb)$ is defined analogously. Let $\cc_{dg}(k)$ be the dg category of differential graded $k$-modules. 
 
\begin{notation}
Let $\cb$ be a dg category. A \emph{right dg $\cb$-module} is a dg functor $L:\cb^{\op} \to \cc_{dg}(k)$. We denote by $\cc_{dg} (\cb)$ the dg category of right dg $\cb$-modules. Denote by $\cc (\cb)$ and $\ch (\cb)$ the categories $Z^0(\cc_{dg} (\cb))$ and $H^0(\cc_{dg} (\cb))$, respectively. The derived category $\cd(\cb)$ is the localization of $\cc(\cb)$ with respect to the quasi-isomorphisms.
\end{notation}

\begin{remark}
\thlabel{exact structure}
The category $\cc (\cb)$ admits an exact structure by defining an \emph{admissible short exact sequence} (or \emph{conflation}) to be a sequence $L\to M\to N$ such that the underlying sequence of graded $\cb$-modules is split short exact. Endowed with this structure, $\cc (\cb)$ becomes a Frobenius category whose stable category is $\ch (\cb)$ (\cf \cite[Lemma 2.2]{Keller deriving dg}).
\end{remark}

\begin{example}
\thlabel{dg cats in deg 0}
Let $\ca$ be an additive category. We consider $\ca$ as a dg category whose morphism complexes are concentrated in degree $0$. Then, the objects of $\cc_{dg}(\ca)$ can be thought of as \emph{complexes of right $\ca$-modules}. The morphism complex
\begin{equation*}
\ch om (X,Y) := \cc_{dg}(\ca)(X,Y)
\end{equation*}
between the complexes $X=\cdots\to X_i\to X_{i+1} \to \cdots$ and $Y=\cdots\to Y_i\to Y_{i+1} \to \cdots$ has as $n^{\text{th}}$ component the $k$-module
\begin{equation}
\ch om^n (X,Y)= \prod_{i\in \Z}\Hom_{\ca}(X_i,Y_{i+n}).
\end{equation}
The differential on $\ch om (X,Y)$ is given by
\begin{equation*}
d(f)=(f_i d_Y- (-1)^n d_X f_{i+1})_{i\in \Z} 
\end{equation*}
for $f=(f_i)_{i\in \Z}\in \ch om ^n(X,Y)$. It follows that $\cc (\ca)$ is the category of chain complexes of $\ca$-modules and that $\ch (\ca)$ is the homotopy category of chain complexes of $\ca$-modules.
\end{example}

\begin{notation}
If $\ca$ is an additive category we denote by $C(\ca)_{dg}$ the dg category of \emph{complexes with components in} $\ca$. Its objects are the complexes with components in $\ca$ and the morphism complex between two objects is defined in the same way as in \thref{dg cats in deg 0}. 
\end{notation}

\begin{warning}
Notice the difference between the categories $\cc_{dg}(\ca)$ and $C(\ca)_{dg}$.
\end{warning}

\begin{remark}
For each dg category $\cb$, the Yoneda functor $\cb \to \cc_{dg}(\cb)$ takes an object $X$ to the representable dg module $\cc_{dg}(\cb)(-,X)$. Slightly abusing, we also call the induced functors $Z^0(\cb)\to \cc(\cb)$ and $H^0(\cb) \to \ch (\cb)$ Yoneda functors.
\end{remark}

\begin{definition}
We call $\cb$ \emph{pretriangulated}, if the image of the Yoneda functor $Z^0(\cb) \hookrightarrow \cc(\cb)$ is closed under translations and extensions (with respect to the exact structure of \thref{exact structure}). If $\cb$ is pretriangulated then the category $Z^0(\cb)$ is a Frobenius subcategory of $\cc (\cb)$. A \emph{dg enhancement} of a triangulated category $\ct$ is a pair $(\cb,\epsilon)$, where $\cb$ is a pretriangulated dg category and $\epsilon: H^0(\cb)\to \ct$ is a triangle equivalence. We call $\ct$ \emph{algebraic} if it admits a dg enhancement.  
\end{definition}

\begin{example} 
\thlabel{example pretriangulated}
Let $\ca$ be an additive category. The dg category $C(\ca)_{dg}$ is a typical example of a pretriangulated dg category. If $\cb$ is a pretriangulated dg category and $\cb'$ is a subcategory of $\cb$ such that $Z^0(\cb')$ is closed in $Z^0(\cb)$ under shifts and extensions, then $\cb'$ is pretriangulated. In particular, if $\ca$ is exact, $\cp$ is a subcategory of  $\ca$ closed under direct sums and we let $\ca c(\cp)_{dg}$ be the full subcategory of $C(\ca)_{dg}$ formed by the acyclic complexes with components in $\cp$, then $\ca c(\cp)_{dg}$ is a pretriangulated dg category.
\end{example}

\begin{example}
\thlabel{canonical enhancement}
Let $\ce$ be a Frobenius category and let $\cp $ be its full subcategory of projective-injective objects. We denote by $\ul{\ce}$ the stable category of $\ce$, \ie the quotient category of $\ce$ by the ideal of morphisms factorizing through elements of $\cp$. By \thref{example pretriangulated}, the dg category $\ca c (\cp)_{dg}$ is pretriangulated. Moreover $\ca c (\cp)_{dg}$ is a (canonical) dg enhancement of $\ul{\ce}$. Indeed: the category $Z^0(\ca c (\cp)_{dg})$ identifies with the category of acyclic complexes with components in $\cp$. Is easy to see that there is triangle equivalence 
\begin{equation*}
Z^0:H^0(\ca c (\cp)_{dg}) \overset{\sim}{\longrightarrow} \ul{\ce}
\end{equation*}
which takes a complex $P^{\cdot}$ to its $0$-cycles $Z^0(P^{\cdot})$ \cf \cite[Section 1.5]{Keller Vossieck}.
\end{example}

\begin{definition}
If $\cb$ is an arbitrary small dg category and $\cb'$ is a dg category, then the category of dg functors $\ch om (\cb, \cb')$ is a dg category in a natural way (see \cite[Section 2.3]{Keller on dg}). There is a universal dg functor $\cb \to \pretr (\cb)$
to a pretriangulated dg category $\pretr (\cb)$, \ie a functor inducing an equivalence of dg categories $\ch om (\cb,\cb')\to \ch om (\pretr (\cb),\cb')$ for each pretriangulated dg category $\cb'$. The dg category $\pretr (\cb)$ is the \emph{pretriangulated hull} of $\cb$.
\end{definition}

\begin{remark}
The pretriangulated hull of $\cb$ is constructed explicitly in \cite{Bondal Kapranov enhanced} (where it is denoted by $\text{Pre-tr}^+(\cb)$), \cf also \cite{Drinfeld} and \cite{Tabuada invariants}.
\end{remark}

\subsection{Quasi-functors.}
Let $\cb$ and $\cb'$ be two dg categories. A $\cb$-$\cb'$-bimodule $M$ is an object of $\cc_{dg}(\cb^{\text{op}}\otimes \cb')$, \ie $M$ is a left $\cb$-module and a right $\cb'$-module. Let $\rep(\cb, \cb')$ be the full subcategory of the derived category $\cd(\cb^{\text{op}}\otimes \cb')$ formed by the bimodules $X$ such that the derived tensor product
\begin{equation*}
?\lten_{\cb} X : \cd(\cb) \to \cd (\cb')
\end{equation*}
takes the representable $\cb$-modules to objects which are isomorphic to a representable $\cb'$-modules. 
\begin{remark}
Every object $X$ in $\cd(\cb^{\text{op}}\otimes \cb')$ is isomorphic to a \emph{cofibrant} object of $\cd(\cb^{\text{op}}\otimes \cb')$. Therefore, in practice we will always assume that every object in $\rep(\cb,\cb')$ is cofibrant and we will consider it as a bimodule in $\cc_{dg}(\cb^{\text{op}}\otimes \cb')$. In particular, we require that $X(B,?)$ is quasi-isomorphic to a representable $\cb'$-module for each object $B$ of $\cb$. 
The category of $\cb'$-modules which are quasi-isomorphic to a representable dg module is equivalent to $H^0(\cb')$. Therefore an object of $\rep(\cb,\cb')$ defines a functor
\begin{equation*}
H^0(\cb) \to H^0(\cb').
\end{equation*}
For this reason, the objects in $\rep(\cb,\cb')$ are called \emph{quasi-functors}. 
\end{remark}
The \emph{bimodule bicategory} $\rep$ has as objects all small dg categories; the morphism category between two objects $\cb$ and $\cb'$ is $\rep(\cb, \cb')$; the composition bifunctor
\begin{equation*}
\rep(\cb', \cb'') \times \rep(\cb,\cb')\to \rep(\cb,\cb'')
\end{equation*}
is given by the derived tensor product $(X,Y)\mapsto X\lten_{\cb'}Y$ . For each dg functor $F:\cb\to \cb'$,we have the dg bimodule
\begin{equation*}
M_{F}:(B,B') \mapsto \cb'(B',F(B))
\end{equation*}
which clearly belongs to $\rep(\cb,\cb')$. 
\\

\subsection{Orbit categories.}

Let $\ca$ be a $k$-linear category and $F: \ca \to \ca$ be an automorphism. By definition, the orbit category $\ca/F$ has the same
objects as $\ca$ and the morphisms between two objects $X$ and $Y$ are given by 
\begin{equation*}
\ca / F(X,Y)= \bigoplus_{l\in \Z} \ca(X,F^l(Y)).
\end{equation*}	
The composition is given by the formula
\begin{equation}
\label{composition}
(f_a)\circ(g_b)=\left(\sum_{a+b=c}F^b(f_a)\circ g_b\right),
\end{equation} 
where $f_a:Y\to F^a(Z)$, $g_b:X\to F^b(Y)$ and $a,b\in \Z$. 
Let $p:\ca \to \ca/F$ be the canonical projection functor. It is endowed with a canonical isomorphism of functors $\phi: p \to p  \circ F$ given by $\phi_X=(\dots, 0,1_{F^{-1}(X)}, 0, \dots) $ for each object $X$ of $\ca$. Let $\ca'$ be another $k$-linear category. An $F$-invariant functor from $\ca$ to $\ca'$ is given by a pair $(G,\phi)$, where $G:\ca \to \ca'$ is a $k$-linear functor and $\phi: G\to G\circ F$ is an isomorphism of functors. A morphism of $F$-invariant functors $(G,\phi)\to (G',\phi')$  is given by a morphism of functors $\alpha : G \to G'$ such that the square
\begin{equation*}
\xymatrix{
G \ar[r]^{\phi}\ar[d]_{\alpha} & G F\ar^{\alpha F}[d]\\
G' \ar[r]^{\phi'} & G'  F 
}
\end{equation*}
commutes. In this way, we obtain the category $\inv_F (\ca,\ca')$ of $F$-invariant functors. In particular, $(p,\phi)$ is an $F$-invariant functor. The orbit category satisfies the following universal property.

\begin{theorem}
\thlabel{universal property}
$($\cite[Corollary 3.5]{Asashiba}. See also \cite{Keller triang corrections}$)$ Let $\fun_k(\ca/F,\ca')$ be the category of $k$-linear functors from $\ca/F$ to $\ca'$. The functor
\begin{equation*}
 \fun_k(\ca/F,\ca') \to \inv_F(\ca,\ca'),\ G\mapsto G \circ p
\end{equation*}
is an isomorphism of categories.
\end{theorem}

\begin{remark}
We may suppose without any risk that $F$ is an autoequivalence. The theoretical justification for this can be found in Section 7 of \cite{Asashiba 2}.
\end{remark}

The dg orbit category associated to a dg category $\cb$ and a quasi-functor in $\rep(\cb,\cb)$ is defined by a universal property. This property can be thought of as a lift of \thref{universal property} to the dg world (see \thref{relation between universal properties} below). To state properly the universal property of the dg orbit categories we need to introduce the following definition. 

\begin{definition}
Suppose that $\cb$ is small and that $F\in \rep(\cb,\cb)$ is given by a cofibrant bimodule. For a dg category $\cb'$, define $\widetilde{\eff}(\cb, F, \cb')$ to be the category whose objects are given by pairs $(P,\phi)$  where
\begin{itemize}
\item $P$ is a quasi-functor in $\rep (\cb,\cb')$,
\item $\phi:P\to P F$ is a quasi-isomorphism of dg bimodules,
\item the morphisms from $(P,\phi)$ to $(P',\phi')$ are obtained the morphisms $f:P \to P'$ of dg bimodules, such that $\phi' \circ f = (f F)\circ \phi$ in the category of dg bimodules. In other words, the following diagram commutes
\begin{equation*}
\xymatrix{
P \ar[r]^{\phi}\ar[d]_{f} & PF \ar[d]^{fF}\\
P' \ar[r]^{\phi'} & P'F.
}
\end{equation*}
\end{itemize}
\end{definition}
Let $\eff (\cb, F, \cb')$ be the localization of $\widetilde{\eff}(\cb, F, \cb')$ with respect to the morphisms $f$ which are quasi-isomorphisms of dg bimodules. 
\begin{remark}
The name $\eff$ comes from the french word \emph{effa\c{c}able} which means erasable. 
\end{remark}
\begin{theorem}$($\cite[Theorem 3 (b)]{Keller triang}$)$
\thlabel{dg universal property}
Let $\cb$ a dg category and $F\in\rep(\cb,\cb)$. Then the $2$-functor
$\eff(\cb,F,?)$ is $2$-representable, \ie there exist a dg category $\cb/F$ and a pair $(P_0, \phi_0)$ in $\eff(\cb,F,\cb/F)$ such that for every small dg category $\cb'$, the functor
\begin{equation*}
 \rep(\cb/ F, \cb') \rightarrow \eff(\cb,F,\cb'),\ G \mapsto G\circ P_0
\end{equation*}
is an equivalence.
\end{theorem} 

We call $\cb/F$ the \emph{dg orbit category} associated to $\cb$ and $F$. If $\cb$ is a dg category endowed with an endomorphism $F:\cb\to \cb$ inducing an equivalence $H^0(F):H^0(\cb)\to H^0(\cb)$, then $\cb/F:=\cb/M_{F}$ can be described explicitly as follows: The objects of $\cb/F$ are the same as the objects of $\cb$. For $X, Y\in \cb/F$, we have
\begin{equation}
\cb / F(X,Y):= \colim_p\bigoplus_{n \geq 0}\cb (F^n(X),F^p(Y)),
\end{equation}
where the transitions maps are given by $F$

\begin{equation*}
\xymatrix{
\displaystyle{\bigoplus_{n\geq 0}}\cb (F^n(X),F^p(Y))\ar^{F}[r] & \displaystyle{\bigoplus_{n\geq 0}}\cb (F^n(X),F^{p+1}(Y)).
}
\end{equation*} 

Combining \thref{dg universal property} with the universal property of the pretriangulated hull we obtain the following universal property:

\begin{theorem}$($\cite[Theorem 4]{Keller triang}$)$
\thlabel{pretr dg universal property}
Let $\cb$ be a pretriangulated dg category and $F\in\rep (\cb,\cb)$. Then for any pretriangulated dg category $\cb'$ there is an equivalence of categories
\begin{equation*}
\xymatrix{
 \rep(\pretr (\cb/F), \cb') \ar^{\phantom{xxl}\sim}[r]& \eff(\cb,F,\cb').
}
\end{equation*}
\end{theorem} 

\begin{remark}
\thlabel{relation between universal properties}
If $(P,\phi)\in \eff(\cb,F,\cb') $ then $H^0(P)$ is an $H^0(F)$-invariant functor. By \thref{universal property} $H^0(P)$ induces a functor $G:H^0(\cb)/H^0(F)\to H^0(\cb')$. By \thref{dg universal property} $(P,\phi)$ corresponds to a quasi-functor $\tilde{G}$ such that $H^0(\tilde{G})=G$.
\end{remark}

The dg orbit category is functorial in $(\cb,F)$ in the following sense.

\begin{lemma}
\thlabel{functoriality of dg orbit}
$($\cf \cite[Section 9.4]{Keller triang}$)$ Let 
\begin{equation*}
\xymatrix{
\cb \ar[r]^X\ar[d]_{F} & \cb' \ar[d]^{F'}\\
\cb \ar[r]^X & \cb' 
}
\end{equation*}
be an square in $\rep$ and let $\gamma : F'X \to XF$ be an isomorphism in $\rep(\cb,\cb')$. Then there is a morphism $\ol{X}: \cb / F \to \cb' / F'$ such that if 
\begin{equation*}
\xymatrix{
\cb' \ar[r]^{X'}\ar[d]_{F'} & \cb'' \ar[d]^{F''}\\
\cb' \ar[r]^{X'} & \cb' 
}
\end{equation*}
is another diagram in $\rep$ as above, then the quasi-functors $\ol{X'}\circ \ol{X}$ and $\ol{X'\circ X}$ are isomorphic. 
\end{lemma}

\begin{definition}
Let $\ct=H^0(\cb)$ be an algebraic triangulated category, and $\widetilde{F}:\cb\to \cb$ a dg functor inducing an equivalence $F:\ct \to \ct$. Then the \emph{triangulated hull} of $\ct/F$ is defined as the triangulated category
\begin{equation*}
H^0(\pretr (\cb/ \widetilde{F})).
\end{equation*}
\end{definition}

\begin{definition}
Let $\ct$ be a triangulated category endowed with an autoequivalence $F:\ct \to \ct $. Suppose that $\ct_{dg}$ is a dg enhancement of $\ct$ and that $\widetilde{F} \in \rep (\ct_{dg},\ct_{dg})$ is a dg lift of $F$. The triangulated hull of $(\ct, F)$ with respect to $\ct_{dg}$ is the triangulated category $H^0(\pretr({\ct_{dg}/\widetilde{F}}))$. We say that $\ct/F$ is triangulated with respect to $(\ct_{dg}, \tilde{F})$ if $\ct/F$ is equivalent to the triangulated hull defined by $\ct_{dg}$ and $\widetilde{F}$.
\end{definition}

\begin{remark}
Let $\ch$ be a hereditary abelian category and $F:\cd^b(\ch)\to \cd^b(\ch)$ an autoequivalence. In \cite{Keller triang}, Keller gives sufficient conditions on $F$ to ensure that $\cd^b(\ch)/F$ is triangulated with respect to $\cd^b(\ch)_{dg}$. 

\end{remark}

\section{The ambient Frobenius category}

\subsection{Frobenius categories and Gorenstein-projective modules} In this subsection we recall a general result on Frobenius categories due to Chen \cite{Chen}. It allows us to embed any Frobenius category $\ce$ into a module category (over an additive category). More precisely, let $\cp $ be the full subcategory of $\ce$ formed by its projective objects, then $\ce$ is equivalent, as an exact category, to an extension-closed exact subcategory of $\gpr(\cp)$, the category of finitely-generated Gorenstein-projective (or maximal Cohen-Macaulay) modules over $\cp$. This theorem is crucial for our construction of an ambient Frobenius category in which the orbit category $\ce/F$ embeds.

\begin{notation}
\thlabel{notation proj}
Let $\ca$ be an additive $\Z$-category. We denote by $\Mod (\ca)$ the category of all right modules over $\ca$ and by $\mod (\ca)$ its full subcategory of finitely presented modules. We let $\proj(\ca)$ be the full subcategory of $\mod (\ca)$ formed by the finitely-generated projective $\ca$-modules. 
\end{notation}

\begin{definition}
\thlabel{gpr}
An $\ca$-module $M$ is \emph{finitely generated Gorenstein projective} if there is an acyclic complex
\begin{equation*}
P_M: \cdots \to P_1 \to P_0 \to P^{0} \to P^{1} \to \cdots
\end{equation*}
of objects in $\proj(\ca)$ such that $M \cong \cok (P_1\to P_0)$ and the complex $\Hom_{\ca}(P_M, P')$ is still acyclic for each module $P'$ in $\proj(\ca)$. Denote by $\gpr (\ca)$ the full subcategory of $\mod (\ca)$ formed by the Gorenstein projective modules. In the situation described above we call $P_M$ a \emph{complete projective resolution} of $M$.
\end{definition}
Notice that every finitely generated projective $\ca$-module $P$ lies in $\gpr(\ca)$, since we may take its complete resolution as $\cdots \to 0 \to P \overset{\sim}{\longrightarrow}  P \to 0 \to \cdots$.
\begin{lemma}
\thlabel{gpr is Frobenius}
The category $\gpr(\ca)$ is a Frobenius category whose subcategory of projective--injective objects is $\proj(\ca)$, the category of finitely-generated projective $\ca$-modules.
\end{lemma}
\begin{proof}
By \cite[Proposition 5.1]{AR} the category $\gpr (\ca)$ is an extension-closed subcategory of $\Mod(\ca)$ and thus, it is an exact category. Let $P$ and $P'$ be finitely-generated projective $\ca$-modules. A complex of the form $(\cdots \to 0 \to P=P \to 0 \to \cdots)$ is acyclic and remains acyclic after applying the functor $\Hom (?, P')$. Therefore $\proj (\ca)$ identifies with the subcategory of projective objects of $\gpr(\ca)$. Let $M$ be a module in $\gpr(\ca)$ and $P_M = (\cdots \to P_1 \to P_0 \to P^{0} \to P^{1} \to \cdots)$ a complete resolution of $M$. Since the complex $\Hom (P_M, P)$ is acyclic, we have that $\Ext_{\ca}^1(M,P)=0$. Therefore $P$ is also injective in $\gpr(\ca)$. The sequences $0\to Z^{-1}(P_M) \to P_0 \to M \to 0 $ and $0\to M\to P^0 \to Z^1(P_M) \to 0 $ are short exact sequences of $\mod(\cp)$ which lie in $\gpr(\cp)$.  Therefore $\gpr(\ca)$ has enough projectives and enough injectives. Moreover each injective object in $\gpr(\ca)$ must be projective too. This completes the proof.
\end{proof}

\begin{definition}
\thlabel{restricted Yoneda}
Let $\cv$ be a subcategory of $\ca$. Then the assignment
\begin{equation*}
X\mapsto \ca(?,X)|_{\cv}
\end{equation*}
induces a functor $\ca \rightarrow \Mod(\cv)$, which we call the \emph{restricted Yoneda functor}. 
\end{definition}

\begin{lemma}
\thlabel{one lemma}
Let $\ce$ be an exact category and let $\cp$ be its full subcategory of projective objects. Let $0\to A \to B \to C \to 0$ be an exact sequence in $\ce$. Then the induced sequence
\begin{equation*}
\xymatrix{
0 \ar[r] & \ce(?,A)|_{\cp} \ar[r] & \ce(?,B)|_{\cp} \ar[r] & \ce(?,C)|_{\cp} \ar[r] &  0
}
\end{equation*}
is exact in $\Mod(\cp)$. 
\end{lemma}
\begin{proof}
This follows from the fact that $\ce(P, ?)$ is exact for every $P\in \cp$.
\end{proof}

\begin{corollary}
\thlabel{Yoneda and gpr}
Suppose $\ce$ is a Frobenius category. Then the essential image of the restricted Yoneda functor $\ce \rightarrow \Mod(\cp)$ is contained in $\gpr(\cp)$.
\end{corollary}
\begin{proof}
It follows from \thref{one lemma} that the complete resolution
\begin{equation*}
\xymatrix{
\cdots \ar[r] & P_1 \ar[rr]\ar[dr] & & P_0 \ar[rr]\ar[dr] & & P^0 \ar[rr]\ar[dr] & & P^1 \ar[r] \ar[rd] & \cdots \\
\cdots \ \ \ar[ru]& & X_{1} \ar[ur] & & X  \ar[ur] & &  X^1  \ar[ur] & & \ \ \cdots 
}
\end{equation*}
obtained by splicing the admissible short exact sequences $0\to  X_{i+1} \to P_{i} \to X_{i} \to 0$ and $0\to  X^{i}  \to P^{i} \to X^{i+1} \to 0$ for $i\geq 0$ and $X_0=X=X^0$ is sent to the complete resolution
\begin{equation*}
\xymatrix{
(?,P_1) \ar[rr]\ar_{\cdots\phantom{jkldsfi dsuh}}[dr] & &(?,P_0) \ar[rr]\ar[dr] & & (?,P^0) \ar[rr]\ar[dr] & & (?,P^1) \\
& (?,X_1)\ar[ur] & & (?,X)  \ar[ur] & &(?,  X^1)  \ar_{\phantom{jkldsfi dsuh}\cdots}[ur] &
}
\end{equation*}
where we abbreviate $\ce(?,?)|_{\cp}$ by $(?,?)$. The claim follows easily.
\end{proof}

The following theorem is a version of \cite[Theorem 4.2]{Chen} that can be deduced from \thref{one lemma} and \thref{Yoneda and gpr}. It allows us to think of $\ce$ as a full exact subcategory of $\gpr(\cp)$.

\begin{theorem}
The restricted Yoneda functor $\ce \to \gpr(\cp) $ is full and faithful. Moreover, its essential image is an exact-closed subcategory of $\gpr(\cp)$.
\end{theorem}

\subsection{The ambient Frobenius category for the orbit category} Suppose that $\ce$ is a essentially small Frobenius category endowed with an exact automorphism $F:\ce \to \ce$. Let $p: \ce \to \ce / F$ be the natural projection. It induces a pair of adjoint functors
\begin{equation*}
\xymatrix{
 \Mod (\ce) \ar_{\pi}@<-0.5ex>[d] \\
 \Mod (\ce / F), \ar_{p^{\ast}}@<-0.5ex>[u]
}
\end{equation*}
where $p^*$ is the \emph{restriction functor} and its left adjoint $\pi$ is the \emph{extension} of $p$ to $\Mod (\ca)$ (\cf \cite[Lemma 2.4]{Krause 98}).  It is clear that the full subcategory of $\ce/F$ defined by the objects in $\cp$ is equivalent to $\cp/F$. Therefore, we will consider $\cp / F $ as a full subcategory of $\ce / F$.

\begin{theorem}
\thlabel{ambient}
The restricted Yoneda functor $\ce/F \to \Mod(\cp/F)$ is full and faithful, its essential image is contained in $\gpr(\cp/F)$ and there is a commutative diagram 
\begin{equation*}
\xymatrix{
\ce \ar@{^{(}->}[r]\ar[d]_{p} & \gpr (\cp)  \ar[d]_{\pi |}   \\
\ce /F \ar@{^{(}->}[r] & \gpr (\cp/F), 
}
\end{equation*}
where $\pi |$ denotes is the restriction of $\pi$ to $\gpr(\cp)$. 
\end{theorem}

Recall that $\ca c(\cp)_{dg}$ is the canonical dg enhancement of $\ul{\ce}$ (see \thref{canonical enhancement}). Let $\widetilde{F}:\ca c(\cp)_{dg} \to \ca c(\cp)_{dg}$ be the dg functor given by $F$ componentwise and let $\ul{F}:\ul{\ce} \to \ul{\ce}$ be the automorphism induced by $F$ on $\ul{\ce}$. The following is a key result to prove our main theorem.

\begin{theorem}
\thlabel{triangulated functor usual}
 Suppose that $\ul{\ce}/\ul{F}$ is triangulated with respect to $(\ca c (\cp)_{dg},M_{\widetilde{F}})$. Then there is full and faithful triangulated functor $\ul{\ce}/\ul{F} \to \ul{\gpr}(\cp / F)$ which makes the following diagram commutative up to isomorphism
\begin{equation*}
\xymatrix{
\ul{\ce} \ar@{^{(}->}[r]\ar[d]_{p} & \ul{\gpr} (\cp)  \ar[d]_{\ul{\pi}}\\
\ul{\ce} /\ul{F} \ar@{^{(}->}[r] & \ul{\gpr} (\cp / F). 
}
\end{equation*}
\end{theorem}


\begin{remark}
\thref{ambient} and \thref{triangulated functor usual} can be stated for \emph{completed orbit categories}. The statements and the proofs in both settings are essentially the same. In section 6 we will give a proof of these theorems for completed orbit categories \cf  \thref{stable embedding} and \thref{important remark}. For this reason, we will omit the proofs of these statements.
\end{remark}


\section{Completed orbit categories}

Let $k$ be a field and $\ca$ an essentially small additive category. Let $F:\ca \rightarrow \ca$ be an automorphism of $\ca$ such that for all objects $X$, $Y$ of $\ca$, the space $\ca (X,F^l (Y))$ vanishes for all integers $l\ll 0$ (whenever we make reference to completed orbit categories we will implicitly assume this condition). We define the \emph{completed orbit category} $\caf$ as the category whose objects are the same as those of $\ca$ and with morphism spaces
\begin{equation}
\ca \corb F(X,Y)= \prod_{l\in \Z} \ca(X,F^l(Y)).
\end{equation}
Notice that the vanishing condition imposed on the spaces $\ca (X,F^l (Y))$ ensures that the composition in $\ca \corb F$ defined as for the usual orbit category in (\ref{composition}) is a well-defined operation. Clearly, the category $\ca \corb F$ is $k$-linear and essentially small. Let $p: \ca \to \caf$ be the natural projection. As before, $p$ induces a pair of adjoint functors
\begin{equation*}
\xymatrix{
 \Mod (\ca) \ar_{\pi}@<-0.5ex>[d] \\
 \Mod (\caf), \ar_{p^{\ast}}@<-0.5ex>[u]
}
\end{equation*}
where $p^*$ is the \emph{restriction functor} and its left adjoint $\pi$ takes a projective module $\ca(?,X)$ to the projective module $\caf(?,p(X))$. We denote by $F_*$ the automorphism $M\mapsto M \circ F^{-1}$ of $\Mod (\ca)$ induced by $F$.

\begin{lemma}
\thlabel{embedding lemma}
\thlabel{projective resolution}
Let $M$ be a finitely presented $\ca$-module. Then
\begin{itemize}
\item[$(i)$] we have a canonical isomorphism
\begin{equation*}
\xymatrix{
p^*\pi (M) \ar^{\phantom{xxx}\sim \px \px \ }[r] &\displaystyle{ \prod_{l\in \Z}} F^l_* (M),
}
\end{equation*}
\item[$(ii)$] let $L$ be an $\ca$-module admitting a resolution $\cdots \to P_1 \to P_0 \to L \to  0$ by finitely-generated projective $\ca$-modules $P_i$. Then the complex
\begin{equation*}
\xymatrix{
\cdots\ar[r] & \pi (P_1) \ar[r] & \pi (P_0) \ar[r] &\pi (L) \ar[r] & 0
}
\end{equation*} 
is a resolution of $\pi (L)$ by finitely-generated projective $\caf$-modules,
\item[$(iii)$] for each $\ca$-module $L$ admitting a resolution by finitely-generated projective $\ca$-modules, there are canonical isomorphisms
\begin{equation*}
\Ext^i_{\Mod(\caf)}(\pi (L),\pi (M)) \cong  \prod_{l\in \Z} \Ext^i_{\Mod(\ca)}( L, F^l_* (M))
\end{equation*}
for all $i\geq 0$,
\item[$(iv)$] if $\ce$ is an exact subcategory of $\mod (\ca)$ stable under the action of $F$, then we have a square commutative up to isomorphism
\begin{equation*}
\xymatrix{
\ce \ar@{^{(}->}[r]\ar[d]_{p} & \Mod (\ca) \ar[d]_{\pi}\\
\eaf\ar@{^{(}->}[r] & \Mod (\caf) 
}
\end{equation*}
with fully faithful horizontal arrows.
\end{itemize}
\end{lemma}
\begin{proof}
$(i)$ Since $\pi (\ca(?,X))=\caf(?,p(X))=\prod_{l\in \Z}F^{l}_{*}(\ca(?,X))$, we have that $p^{*}\pi (P)=\prod_{l\in \Z}F^{l}_{*}(P)$ for all projective modules of finite type. Since $p^*\pi$ is right exact, we have
\begin{equation}
\label{p of pi}
p^{*}\pi (M)=\prod_{l\in \Z}F^{l}_{*}(M)
\end{equation}
for all $M$ in $\mod (\ca)$.

$(ii)$
By $(i)$ the complex $\cdots\rightarrow p^*\pi (P_1) \rightarrow p^*\pi (P_0) \rightarrow p^*\pi (L) \rightarrow 0$ is exact and for all $l\geq 0$,  $\pi (P_l)$ is a finitely-generated projective $\caf$-module. Since $p^*$ is the restriction functor, the claim follows.

$(iii)$ 
We have the following isomorphisms 
\begin{align}
\label{iso 1}
\Hom_{\caf}(\pi (L),\pi (M)) \cong&\ \Hom_{\ca}(L,p^*\pi (M)) \nonumber \\
\cong&\ \Hom_{\ca}(L,\prod_{l\in \Z}F^{l}_{*}(M)) \\ 
\cong& \prod_{l\in \Z}\Hom_{\ca}(L,F^{l}_{*}(M)).\nonumber  
\end{align} 
By $(ii)$, the complex $ \cdots\rightarrow \pi (P_1) \rightarrow \pi (P_0) \rightarrow \pi L \rightarrow 0$ is a projective resolution of $\pi (L)$. After applying the functor $\Hom_{\caf}(?,\pi (M)) $ to this resolution and the last isomorphism in \eqref{iso 1} the claim follows.\\
$(iv)$ This is immediate from \eqref{iso 1}.
\end{proof}

\begin{lemma}
\thlabel{Krull-Schmidt}
\begin{itemize}
\item[$(i)$] If $X$ is an object of $\ca$ such that $\ca (X,X)$ is local and $\ca (X,F^l(X))$ vanishes for all $l<0$, then $\caf(p(X),p(X))$ is local.
\item[$(ii)$] If $\ca$ is a Krull-Schmidt category such that for each indecomposable object $X$, the ring $\ca(X,X)$ is local and $\ca(X,F^{l}(X))$ vanishes for all $l<0$, then $\caf$ is a Krull-Schmidt category whose indecomposables are the images of those of $\ca$.
\end{itemize}
\end{lemma}
\begin{proof}
$(i)$ We can easily see that $(f_i)\in \caf(p(X),p(X))$ is non-invertible if and only if $f_0:X\to X$ is non-invertible. This shows that $\caf(p(X),p(X))$ is local.

$(ii)$ By part $(i)$, the image $p(X)$ of each indecomposable $X$ of $\ca$ is indecomposable with local endomorphism ring. Hence, since each object of $\ca$ decomposes into a sum of indecomposables, the same holds for $\caf$.
\end{proof}

Suppose that $\ca$ has enough projectives and let $\cp$ denote its full subcategory of projective objects. The essential image of $\ca$ under $p$ is canonically identified with $\paf$. Let $\langle \paf \rangle$ be the ideal of morphisms of $\caf$ which factor through $\paf$. Denote by $\ul{F}$ the automorphism in $\ul{\ca}$ induced by $F$. The canonical projection $p: \ca \to \caf$ induces an $\ul{F}$-invariant functor $\ul{\ca}\longrightarrow \caf/\langle  \paf \rangle$. By the universal property of orbit categories we obtain a functor 
$\psi:\ul{\ce}/\ul{F}\longrightarrow \eaf/\langle  \paf \rangle$. 

\begin{proposition}
\thlabel{additive equivalence}
The functor $\psi$ is faithful. Moreover, if for all objects $X$, $Y$ of $\ca $ we have that $\ul{\ca}(X,\ul{F}^l(Y))=0$ for $l \gg 0$, then $\psi$ is fully.
\end{proposition}
\begin{proof}
For simplicity, along the proof we will denote $\caf  / \langle  \paf \rangle$ by $\ul{\caf}$. Let $P \overset{f}{\rightarrow} Y$ be a projective cover in $\ca$. Then the morphism
\begin{equation*}
\xymatrix{
F^l(P) \ar^{F^l (f)}[r] & F^l (Y)
}
\end{equation*}
is a projective cover for all $l\in \Z$. In particular, every morphism $p(P')\to p(Y)$ in $\caf$ with $P'\in \cp$ can be factorized through $p(f)$. Moreover, since $p(f)$ is concentrated in one degree we obtain the following isomorphisms
\begin{align}
\label{iso 1}
\ul{\caf}(p(X),p(Y)) \cong&\ \cok\left(\caf(p(X),p(P))\to \caf(p(X),p(Y))\right) \nonumber  \\
\cong&\ \cok\left(\prod_{l\in \Z}\ca(X,F^l(P))\to \prod_{l\in \Z}\ca(X,F^l(Y))\right) \nonumber  \\
\cong&\ \prod_{l\in \Z}\left(\cok \left(\ca(X,F^l(P))\to \ca(X,F^l(Y)\right)\right) \nonumber  \\ 
\cong&\ \prod_{l\in \Z} \ul{\ca}(X,\ul{F}^l(Y))\nonumber \\
=&\ \ul{\ca} \corb \ul{F} (X,Y).  \nonumber  
\end{align} 
We consider $\ul{\ca}/\ul{F} $ as a subcategory of $\ul{\ca} \corb \ul{F}$. One checks that, up to equivalence, the functor $\ul{\ca}/\ul{F}\to \ul{\caf}$ is given by this chain of isomorphisms at the level of morphisms. Both claims follow.
\end{proof}

\begin{remark}
\thlabel{important remark}
\thref{embedding lemma} and \thref{additive equivalence}  can be stated for usual orbit categories. The statements and the proofs are essentially the same. Notice that for usual orbit categories the functor $\psi:\ul{\ca}/\ul{F}\longrightarrow \ca /F /\langle  \cp /F \rangle $ is always full and faithful.
\end{remark}

\section{Comparison of the triangulated structures}
We use the universal property of dg orbit categories to construct a triangulated functor which will be crucial in the proof of the main theorem. Let $\ce$ be a Frobenius category endowed with an exact automorphism $F:\ce \to \ce$. 
The restriction of $F$ to $\cp$ induces an automorphism of $\cp$. We denote by $\widetilde{F}:\ca c(\cp)_{dg} \to \ca c(\cp)_{dg}$ the dg functor given by $F$ componentwise. 
Notice that $M_{\widetilde{F}}$ induces a triangle functor on $\ul{\ce}$ which is equivalent to $\ul{F}$. We identify $\ce$ with an exact subcategory of $\Mod(\cp)$ via the restricted Yoneda embedding. The adjoint functors 
$\pi : \Mod(\cp)\mathrel{\mathop{\rightleftarrows}} \Mod(\paf):p^*$ induce a pair of adjoint functors
\begin{equation*}
\xymatrix{
 \cc_{dg}(\cp) \ar_{\tilde{\pi}}@<-0.5ex>[d]\\
\cc_{dg}(\paf).  \ar_{\tilde{p}^{\ast}}@<-0.5ex>[u]. 
}
\end{equation*}
defined componentwise. We consider $\ca c(\proj (\cp))_{dg}$ (resp. $\ca c(\proj (\paf))_{dg}$) as a full subcategory of $\cc_{dg}(\cp)$ (resp. $\cc_{dg}(\paf)$), see \thref{canonical enhancement} and \thref{notation proj}. 

\begin{lemma}
\thlabel{pi on acyclic}
The functor $\tilde{\pi}$ restricts to a functor
\begin{equation*}
\xymatrix{
\tilde{\pi}: \ca c(\proj (\cp))_{dg} \to \ca c(\proj (\paf))_{dg}.
}
\end{equation*}
\end{lemma}

\begin{proof}
Let $P^{\cdot}: \cdots \to P_{-1}\to P_0 \to P_1\to \cdots$ be a complex in $\ca c(\proj (\cp))_{dg}$. Then the complex $\tilde{\pi} (P^{\cdot})= \cdots \to \pi(P_1) \to \pi(P_0) \to \pi(P_{-1}) \to \cdots$ is a complex of finitely-generated projective $\paf$-modules. For each $i\in \Z$ consider the truncated complex 
\begin{equation*}
 \cdots \to \pi(P_{i-1})\to \pi(P_i) \to M \to 0,
\end{equation*}
\ie $M=\cok (\pi(P_{i-1})\to \pi(P_i))$. It follows from \thref{projective resolution} that this complex is acyclic. Since $i \in \Z$ is arbitrary then $\pi (P^{\cdot})$ is acyclic. 
\end{proof}

\begin{lemma}
\thlabel{restriction to gpr}
The functor $\pi: \Mod(\cp)\to \Mod(\paf)$ restricts to a functor $\pi:\gpr(\cp)\to \gpr(\paf)$.
\end{lemma}
\begin{proof}
Let $M$ be a module of $\gpr(\cp)$. In particular, there are objects $(P_i)_{i\in\Z}$ of $\cp$ and an acyclic complex 
\begin{equation*}
P^{\cdot}:\cdots \to \cp(?,P_i)\to \cp(?,P_{i-1}) \to \cdots 
\end{equation*}
such that $M\cong Z^0 (P^{\cdot})$. By \thref{pi on acyclic}, the complex 
\begin{equation*}
\tilde{\pi} (P^{\cdot})= \cdots \to \cp(?,p(P_i))\to \cp(?,p(P_{i-1})) \to \cdots 
\end{equation*}
is acyclic and $\pi(M)\cong Z^0(\pi(P^{\cdot}))$. If $X$ and $Y$ are arbitrary objects of $\cp$ then, by part $(i)$ of \thref{embedding lemma} and the Yoneda lemma, we obtain an isomorphism 
\begin{equation*}
\Hom(\paf(?,p(X)),\paf(?,p(Y)))\cong\prod_i\cp(X,F^{l}_{\ast}(Y)).
\end{equation*}
It follows that  $\Hom (\pi(P^{\cdot}), P')$ is still acyclic for each module $P'$ in $\proj(\paf)$.
\end{proof}

The functor $\pi:\gpr(\cp)\to \gpr(\paf)$ preserves projectives, therefore it induces a functor $\ul{\pi}:\ul{\gpr}(\cp)\to \ul{\gpr}(\paf)$.
\begin{corollary}
\thlabel{stable embedding}
There is a square commutative up to isomorphism
\begin{equation*}
\xymatrix{
\ul{\ce} \ar[r]\ar[d]_{p} & \ul{\gpr} (\cp)  \ar[d]_{\ul{\pi}}\\
\ul{\ce} /\ul{F} \ar[r] & \ul{\gpr} (\paf), 
}
\end{equation*}
with faithful horizontal arrows. If for all $X$ and $Y$ we have that $\ul{\ce}(X,F^l(Y))=0$ for $l \gg 0$, then $\ul{\ce} /\ul{F} \to  \ul{\gpr} (\paf)$ is fully faithful.
\end{corollary}
\begin{proof}
By \thref{embedding lemma} and \thref{restriction to gpr} there is a commutative diagram
\begin{equation*}
\xymatrix{
\ce \ar@{^{(}->}[r]\ar[d]_{p} & \gpr (\cp)  \ar[d]_{\pi}\\
\eaf \ar@{^{(}->}[r] & \gpr (\paf).
}
\end{equation*}
with faithful horizontal arrows. The functor $\eaf \hookrightarrow \gpr(\paf)$ identifies $\paf$ with $\proj (\paf)$. Therefore there is a commutative diagram
\begin{equation*}
\xymatrix{
\ul{\ce} \ar[r]\ar[d] & \ul{\gpr} (\cp)  \ar[d]_{\ul{\pi}}\\
\eaf/ \langle\paf \rangle \ar[r] & \ul{\gpr} (\paf).
}
\end{equation*}
The claim follows from \thref{additive equivalence}.
\end{proof}

The functor $ \ul{\ce}/\ul{F} \to \ul{\gpr} (\paf) $ corresponds to the $\ul{F}$-invariant functor $G$ given by the composition
\begin{equation*}
G: \ul{\ce}\hookrightarrow \ul{\gpr} (\cp)\to \ul{\gpr}(\paf). 
\end{equation*}
We denote by $\widetilde{G}$ the composition
\begin{equation*}
\xymatrix{
\widetilde{G}: \ca c (\cp)_{dg} \ar[r] & \ca c (\proj(\cp))_{dg} \ar^-{\tilde{\pi}}[r]  & \ca c (\proj (\paf))_{dg},
}
\end{equation*}
where the first arrow is given by applying the Yoneda functor $\cp \hookrightarrow \proj(\cp)$ componentwise. Notice that $\ca c (\proj(\cp))_{dg}$ is a dg enhancement of $ \ul{\gpr} (\cp)$ and that $H^0(\widetilde{G})=G$.
\begin{theorem}
\thlabel{triangulated functor}
Suppose that $\ul{\ce}/\ul{F}$ is triangulated with respect to $(\ca c (\cp)_{dg},M_{\widetilde{F}})$. Then the functor $\ul{\ce}/\ul{F} \to \ul{\gpr}(\paf)$ is triangulated.
\end{theorem}
\begin{proof}
It is enough to prove that there is a quasi-isomorphism of quasi-functors $\phi:\widetilde {G}\to\widetilde {G} M_{\widetilde{F}} $ such that $(\widetilde{G},\phi)\in \eff (\ca c (\cp)_{dg}, M_{\widetilde{F}}, \ca c(\proj(\paf))$. In other words, we have to prove that for all complexes $P^{\cdot} \in \ac(\cp)_{dg}$ and $Q^{\cdot} \in \ac(\proj(\paf))_{dg}$, there is a quasi-isomorphism 
\begin{equation*}
\xymatrix{
\ch om (Q^{\cdot},\widetilde{G}(P^{\cdot})) \ar[r] & \ch om (Q^{\cdot},\widetilde{G}\widetilde{F}(P^{\cdot})).
}
\end{equation*}
The complex $P^{\cdot}$ is of the form 
\begin{equation*}
\cdots \longrightarrow P_i \longrightarrow P_{i+1} \longrightarrow \cdots,
\end{equation*}
where $P_i$ is object in $\cp$ for each $i\in \Z$. Then the complex $\widetilde{G}(P^{\cdot})$ is given by
\begin{equation*}
\cdots \longrightarrow \paf(?,p(P_i)) \longrightarrow \paf(?,p(P_{i+1})) \longrightarrow \cdots
\end{equation*}
and the complex $\widetilde{G}\widetilde{F}(P^{\cdot})$ is given by
\begin{equation*}
\cdots \longrightarrow \paf(?,pF(P_i)) \longrightarrow \paf(?,pF(P_{i+1})) \longrightarrow \cdots.
\end{equation*}
For each $i \in \Z$ the objects $P_i$ and $F(P_i)$ are isomorphic in $\paf$. Therefore there is an isomorphism $ \paf(?,p(P_i))\to  \paf(?,pF(P_i))$. In particular, the complexes $\widetilde{G}(P^{\cdot})$  and $\widetilde{G}\widetilde{F}(P^{\cdot})$ are isomorphic. The claim follows.
\end{proof}

\section{The main theorems}
We keep the notation of the preceding section and suppose that $\ul{\ce}/\ul{F}$ is equivalent to its triangulated hull with respect to $(\ca c(\cp)_{dg},M_{\widetilde{F}})$. 

\begin{theorem}
\thlabel{main theorem}
Suppose that $\ce$ is Krull-Schmidt such that
\begin{itemize}
\item for each indecomposable object $X$, the ring $\ce(X,X)$ is local and $\ce(X,F^{l}(X))$ vanishes for all $l<0$,
\item for every pair of objects $Y$ and $Z$, the space $\ul{\ce}(Y,\ul{F}^l (Z))$ vanishes for $l\gg 0$. 
\end{itemize}
Then $\eaf$ admits the structure of a Krull-Schmidt Frobenius category which makes the canonical projection $\ce \to \eaf$ exact and whose stable category is triangle equivalent to $\ul{\ce}/\ul{F}$.
\end{theorem}

\begin{proof}
By part $(iv)$ of \thref{embedding lemma}, \thref{restriction to gpr} and \thref{stable embedding} there is a commutative diagram
\begin{equation*}
\xymatrix{
\ce  \ar[d]_p\ar@{^{(}->}[r] & \gpr(\cp) \ar_{\pi}[d]\\
\eaf \ar[d] \ar@{^{(}->}[r] & \gpr(\paf) \ar[d]\\
\ul{\ce}/\ul{F} \ar[r] & \ul{\gpr}(\paf).
}
\end{equation*}
The first two horizontal arrows are full and faithful. So we may think of $\ce $ (resp. $\eaf)$ as a full subcategory of $\gpr(\cp)$ (resp. $\gpr(\paf)$). In particular, we identify $p$ with the restriction of $\pi$ to $\ce$. 
We first show that $\eaf $ is an extension-closed subcategory of $\gpr(\paf)$. 
Let $0\to \pi(X) \to E \to \pi(Y) \to 0$ be an extension in $\gpr(\paf)$ 
between two objects of $\eaf$. 
By part $(iii)$ of Lemma \ref{embedding lemma},
we have an isomorphism 
\begin{align*}
\Ext^1_{\gpr(\paf)}(\pi(Y),\pi(X)) &\cong \prod_{l\in \Z}\Ext^{1}_{\gpr(\cp)}(Y,F^l_{\ast}(X)) \\
& \cong \prod_{l\in \Z}\Ext^{1}_{\ce}(Y,F^l(X)).
\end{align*} 
Since $\ul{\ce}(X,\ul{F}^l(Y))=0$ for $l\gg 0$ the categories 
$\ul{\ce} \corb \ul{F}$ and $\ul{\ce}/\ul{F}$ are isomorphic. 
In particular, there is an isomorphisms $\displaystyle{\prod_{l\in \Z}}\Ext^{1}_{\ce}(Y,F^l(X)) \cong \Ext^{1}_{\ul{\ce}/\ul{F}}(Y,F^l(X))$
and therefore, an isomorphism 
\begin{equation}
\label{bijection}
\Ext^1_{\gpr(\paf)}(\pi(Y),\pi(X)) \cong\Ext^{1}_{\ul{\ce}/\ul{F}}(Y,X).
\end{equation}
Let
\begin{equation*}
\xymatrix{
X \ar[r] & E' \ar[r] &Y  \ar[r] & \Sigma X
}
\end{equation*}
be the triangle in $\ul{\ce}/\ul{F}$ corresponding to the extension $0\to \pi(X)\to E \to \pi(Y) \to 0$ under (\ref{bijection}). By \thref{triangulated functor}, there is a commutative diagram in $\ul{\gpr}(\paf)$
\vspace{2mm}
\begin{equation*}
\vspace{2mm}
\xymatrix{
\pi(X) \ar[r]\ar[d]_{\cong} &E\ar@{.>}[d]_{a} \ar[r]& \pi(Y )\ar[r] \ar[d]_{\cong} & \Sigma \pi(X) \ar[d]_{\cong}  \\
\pi(X) \ar[r]& \pi(E') \ar[r] &  \pi(Y) \ar[r]  & \Sigma \pi (X).
}
\end{equation*}
The first row corresponds to the triangle associated to the extension $0\to \pi(X)\to E \to \pi(Y) \to 0$, the second row is the image of the triangle $X\to E'\to Y \to \Sigma X$ under the triangle functor $\ul{\ce}/\ul{F}\to \ul{\gpr}(\paf)$. Notice that the square on the right commutes by the choice of $X\to E'\to Y \to \Sigma X$. Therefore we obtain the arrow $E \overset{a}{\to} \pi (E')$ by the axiom of triangulated categories. The Yoneda functor sends each of the rows of this diagram to a long exact sequence (to the left). We apply the five lemma to conclude that $a$ is an isomorphism. Thus, $\pi(E)$ is a direct factor of the sum of $E'$ with a finitely-generated projective module. By \thref{Krull-Schmidt}, we have that $\eaf$ is Krull-Schmidt. In particular, it has split idempotents and it follows that $\pi(E)$ lies in $\ce \corb F$. Hence $\eaf$ is stable under extensions in $\gpr (\cp \corb F)$. Notice that the projective-injective objects of $\gpr(\paf)$ belong to $\eaf$ and are precisely the image of the projectives of $\ce$ under the right exact functor $\pi$. Therefore $\eaf$ is an exact category with enough projectives and such that projectives are injectives. Since $\eaf$ is closed under $\Sigma$ in $\gpr(\paf)$ it also has enough injectives. Now it is easy to see that $\ul{\eaf}$ is triangle equivalent to $\ul{\ce}/\ul{F}$. The exactness of the canonical functor $\ce \to \eaf$ follows from the fact that the induced functor $\ul{\ce} \to \ul{\ce}/\ul{F} $ is triangulated.
\end{proof}

We can use  \thref{ambient}, \thref{triangulated functor usual} and \thref{important remark} to prove the following theorem. The proof is essentially the same as the proof of \thref{main theorem}.

\begin{theorem}
\thlabel{main theorem usual}
Suppose that $\ce/F$ has split idempotents. Then $\ce /F$ admits the structure of a Frobenius category which makes the canonical projection $\ce \to \ce/F$ exact and whose stable category is triangle equivalent to $\ul{\ce}/\ul{F}$.
\end{theorem}

\begin{remark}
If we want to use the results in \cite{Keller triang} to determine if $\ul{\ce}/\ul{F}$ is triangulated we may assume that $\ce$ is $\Ext$-finite, since Keller considers $\Hom $-finite triangulated categories.
\end{remark}

\begin{remark}
We don not have a criterion to determine if $\ce / F$ has split idempotents.
\end{remark}

\section{Applications to cluster algebras (sketch)}

In this section we use the theory of Nakajima categories introduced in \cite{Keller Scherotzke 1}. The reader is referred to this article for some of the definitions and relevant background. We can use \thref{main theorem} to obtain explicit categorifications of families of finite-type cluster algebras with coefficients. In particular, we obtain a categorification of all skew-symmetric finite-type cluster algebras with universal coefficients.

\subsection{Nakajima categories}

Let $Q$ be a finite and acyclic quiver. The \emph{repetition quiver} (\cf \cite{Riedtmann80b}) 
$\Z Q$ is the quiver obtained from $Q$ as follows
\begin{itemize}
\item the set of vertices of $\Z Q$ is $(\Z Q)_0= Q_0 \times \Z$.
\item For each arrow $\alpha : i\longrightarrow j$ of $Q$ and each $p\in \Z$, $\Z Q$ has the arrows
\begin{equation*}
\xymatrix{
(\alpha,p):(i,p) \ar[r]  & (j,p) & \text{and} &  \sigma(\alpha, p):(j,p-1) \ar[r] & (i,p).
}
\end{equation*}
\item $\Z Q$ has no more arrows than the ones described above. 
\end{itemize}

There is a bijection on the set of arrows $\sigma:(\Z Q)_1\to (\Z Q)_1$ given by
\begin{equation*}
\sigma (\beta)=
\begin{cases}
\sigma(\alpha,p) & \text{if  } \beta=(\alpha,p), \\
(\alpha, p-1) & \text{if  } \beta=\sigma(\alpha,p).
\end{cases}
\end{equation*}
Let $\tau : \Z Q \to \Z Q$ be the automorphism of $\Z Q$ given by \emph{the translation by one unit}:
\begin{equation*}
\xymatrix{
\tau(i,p)=(i,p-1)    & \text{and} & \tau(\beta)=\sigma^2(\beta) 
}
\end{equation*}
for each vertex $(i,p)$ and each arrow $\beta$ of $\Z Q$. 

Let $k$ be a field. 
Following \cite{Gabriel80} and \cite{Riedtmann80}, we define the {\em mesh category $k(\Z Q)$} to be the quotient of the path category $k\Z Q$ by the ideal generated by the mesh relators, \ie the $k$-category whose objects are the vertices of $\Z Q$ and whose morphism space from $a$ to $b$ is the space of all $k$-linear combinations of paths from $a$ to $b$ modulo the subspace spanned by all elements $u r_x v$, where $u$ and $v$ are paths and 
\[
r_x = \sum_{\beta: y \to x} \beta \sigma(\beta): \quad
\raisebox{1.225cm}{\xymatrix@R=0.5cm@C=0.5cm{  & y_1 \ar[dr]^{\beta_1} & \\
\tau(x) \ar[ur]^{\sigma(\beta_1)} \ar[dr]_{\sigma(\beta_s)}  & \vdots & x \\
 & y_s \ar[ur]_{\beta_s} & }}
\]
is the {\em mesh relator} associated with a vertex $x$ of $\Z Q$. Here the sum runs over all arrows $\beta: y \to x$ of $\Z Q$.

\begin{definition}
The \emph{framed quiver} $\tilde{Q}$ associated to $Q$ is the quiver obtained from $Q$ by adding, for each vertex $i$, a new vertex $i'$ and a new arrow $i\to i'$. We call the vertices in $(\Z \tilde{Q})_0$ of the form $(i',n)$, $i\in Q_0$, $n\in \Z$, \emph{frozen vertices}. The \emph{regular Nakajima category} $\cR$ associated to $Q$ is the quotient of the path category $k\Z \tilde{Q}$ by the ideal generated by the mesh relators associated to the non-frozen vertices. The \emph{singular Nakajima category} $\cs$ is the full subcategory of $\mathcal{R}$ whose objects are the frozen vertices. 
\end{definition} 

\begin{remark}
Note that there is a bijection $\sigma: (\Z \tilde{Q})_0\to (\Z \tilde{Q})_0$ given by $\sigma: (i,n)\mapsto (i',n-1)$ and $(i',n)\mapsto (i,n)$ for $i$ a vertex of $Q$ and $n$ an integer.
\end{remark}

\begin{assumption}
From now on we let $Q$ be an orientation of a simply laced Dynkin diagram $\Delta$ and denote the bounded derived category of $\mod (kQ)$ by $\cd^b_Q$.
\end{assumption}

\begin{definition}
Let $C$ be a subset of $(\Z Q)_0$. Denote by $\cR_C$ the quotient of $\cR$ by the ideal generated by the identities of the frozen vertices not belonging to $\sigma^{-1}(C)$ and by $\cs_C$ its full subcategory formed by the vertices in $\sigma^{-1}(C)$. We call $C$ an \emph{admissible configuration of} $\Z Q$ if the sequences 
\begin{equation} \label{left-exact-sequences}
0 \to \cR_C(?,x) \to \bigoplus_{x \to y} \cR_C(?,y) \quad \mbox{and}\quad
0 \to \cR_C(x,?) \to \bigoplus_{y \to x} \cR_C(y,?) 
\end{equation}
are exact, where the sums range over all arrows of $\Z \tilde{Q}$ whose source (respectively, target) is $x$. We denote by $\Z \tilde{Q}_C$ the quiver obtained from $\Z \tilde{Q}$ by removing the vertices not belonging to $\sigma^{-1} (C)$. We refer the reader to Section 3.3 of \cite{Keller Scherotzke 1} for an account of sufficient conditions on $C$ in which \eqref{left-exact-sequences} holds.
\end{definition}

\begin{theorem} $($\cite{Keller Scherotzke 1}$)$
\thlabel{description of gpr}
Let $C$ be an admissible configuration of $\Z Q$. Then 
\begin{itemize}
\item[$(i)$] the restriction functor $\Mod \cR_C\to \Mod \cs_C$ induces an equivalence between the full subcategory of finitely generated projective $\cR_C$-modules $\proj (\cR_C)$ and the category $\gpr(\cs_C)$. In particular, it yields an isomorphism of $\Z \tilde{Q}_C$ onto the Auslander-Reiten quiver of $\gpr(\cs_C)$ so that the vertices of $\sigma^{-1}(C)$ correspond to the projective--injective objects,
\item[$(ii)$] there is a triangle equivalence $\Phi: \ul{\gpr}(\cs_C) \to \cd^b_Q$.
\end{itemize}
\end{theorem}

\subsection{Categorification of cluster algebras with coefficients}
Let $\tau:\cd^b_Q \to \cd^b_Q $ be the \emph{Auslander-Reiten translation} and $\Sigma : \cd^b_Q \to \cd^b_Q$ be the suspension functor of $\cd^b_Q$. The cluster category $\cc_Q$ was introduced in \cite{BMRRT} and is defined as the orbit category 
\begin{equation*}
\cc_Q= \cd^b_Q/\Sigma \circ \tau^{-1}.
\end{equation*}
The category $\cc_Q$ is triangulated. Its triangulated structure comes from the dg category $C^b(\proj kQ)_{dg}$.

We denote $\Sigma \circ \tau^{-1}$ by $F_{\cd}$. By a well know result of Happel \cite{Happel derived}, we know that the Auslander-Reiten quiver of the category $\cd^b_Q$ is canonically isomorphic to $\Z Q$. Therefore, $F_{\cd}$ induces an automorphism of translation quivers $F:\Z Q \to \Z Q$. Let $C \subset \Z Q$ be an admissible configuration which is invariant under $F$. By \thref{description of gpr}, $F$ induces a functor 
\begin{equation*}
F_{\ast}:\gpr (\cs_C)\to \gpr (\cs_C)
\end{equation*}
which is in fact exact. \thref{main theorem} is fundamental to prove the following results.

\begin{lemma}
\thlabel{Frobenius model}
The completed orbit category $\gpr(\cs) \corb F_{\ast}$ admits the structure of a Frobenius category whose stable category is triangle equivalent to $\cc_Q$.
\end{lemma}
 
\begin{theorem}
\thlabel{categorification}
$(i)$
Let $C\subset \Z Q$ be an admissible configuration which is invariant under $F$. Then  $\gpr(\cs_C)\corb F$ is a 2-Calabi-Yau realization $($in the sense of \cite{BIRS}$)$ of a cluster algebra with coefficients of type $\Delta$,

$(ii)$ if $C =\Z Q$ then  $\gpr(\cs_C)\corb F$ is a 2-Calabi-Yau realization of the cluster algebra with universal coefficients of type $\Delta$. 
\end{theorem}

The proofs of \thref{Frobenius model} and \thref{categorification}, the details of this section and further applications will be given in a forthcoming paper \cite{Najera Universal}.

\end{document}